\documentclass{article}
\usepackage{amsmath, amssymb, amsthm}
\usepackage{graphicx}
\usepackage{geometry}
\usepackage{tikz}
\usepackage[utf8]{inputenc}
\usepackage{cite}
\usepackage{subcaption}
\usepackage[normalem]{ulem}
\usepackage{hyperref}
\usepackage{pgfplots}
\pgfplotsset{compat=1.16} 
\usepackage{color} 
\usepackage{mathtools} 
\usepackage{algorithm}
\usepackage{algpseudocode}
\hypersetup{
    colorlinks=true,
    linkcolor=black,
    citecolor=black,
    urlcolor=black
}

\usetikzlibrary{
    positioning,    
    fit,            
    shapes.geometric, 
    shapes.misc,    
    decorations.pathmorphing 
}

\newtheorem{theorem}{Theorem}
\newtheorem{lemma}{Lemma}
\newtheorem{claim}{Claim}
\newtheorem{proposition}{Proposition}
\theoremstyle{definition}
\newtheorem{definition}{Definition}

\newtheorem{remark}{Remark}

\newcommand{\defi}{\mathrm{defi}}

\newcommand{\mc}{\mathcal{C}}

\newcommand{\sat}{\mathrm{sat}}
\newcommand{\Sat}{\mathrm{Sat}}

\title{Universal Vertices and Saturation Numbers for Disjoint Triangles}
\author{%
  Xiaoteng Zhou\thanks{Corresponding author. Email: \texttt{zxt@amp.i.kyoto-u.ac.jp}.}\\
  Graduate School of Informatics, Kyoto University, Japan\\
  \texttt{zxt@amp.i.kyoto-u.ac.jp}
  \and
  Naoyuki Kamiyama\\
  Graduate School of Informatics, Kyoto University, Japan\\
  \texttt{kamiyama@amp.i.kyoto-u.ac.jp}
}

\begin{document}

\maketitle

\begin{abstract}
A graph \(G\) is \(\mathcal F\)-saturated if \(G\) contains no member of
\(\mathcal F\), but the addition of any non-edge creates a copy of a
member of \(\mathcal F\).  For \(m\ge 1\), let \((m+1)K_3\) denote the
vertex-disjoint union of \(m+1\) triangles. 
In this paper, 
we study \((m+1)K_3\)-saturated graphs.
We construct a family of \((m+1)K_3\)-saturated graphs which gives
the uniform upper bound
$\sat((m+1)K_3,n)=O(n^{3/2})$
for all \(m\ge 1\) and \(n\ge 3m+3\).
For general $(m+1)K_p$-saturated graphs, 
Faudree et al. determined $\sat((m+1)K_p, n)$ for sufficiently large $n$.
In the case of $p=3$, 
we lower Faudree's threshold from $n \ge 12m + 3$ to \(n\ge 9m+5\) and prove that
$\sat((m+1)K_3,n)=n+6m-1.$
We also provide two structural restrictions on the components obtained after
deleting all vertices of degree $|V(G)|-1$ from an \((m+1)K_3\)-saturated graph.
\end{abstract}

\section{Introduction}
\label{intro}

A central theme of graph theory is the study of extremal problems,
originating with the work of Turán~\cite{Turan1941}. Given a family \(\mathcal F\) of graphs, 
the maximum number of edges in an \(n\)-vertex graph containing no member of
\(\mathcal F\) as a subgraph is denoted by
\[
    \operatorname{ex}(\mathcal F, n)
    =
    \max\{|E(G)|: |V(G)|=n,\ G\text{ contains no }F\in\mathcal F
    \text{ as a subgraph}\}.
\]

In this paper, we study the corresponding saturation problem. A graph
\(G=(V,E)\) is called \(\mathcal F\)-saturated if \(G\) contains no member
of \(\mathcal F\) as a subgraph, but for every edge
\(e\in E(\overline G)\), the graph \(G+e\) contains some member of
\(\mathcal F\) as a subgraph, where \(\overline G\) denotes the complement
of \(G\).

From the viewpoint of saturated graphs,  the \emph{extremal number} can equivalently be written
as
\[
    \operatorname{ex}(\mathcal F, n)
    =
    \max\{|E(G)|: |V(G)| = n \text{ and } G \text{ is } \mathcal{F}\text{-saturated}\}.
\]
The \emph{saturation number} is the corresponding minimum:
\[
    \operatorname{sat}(\mathcal F,n)
    =
    \min\{|E(G)|: |V(G)| = n \text{ and } G \text{ is } \mathcal{F}\text{-saturated}\}.
\]
Similarly,  $\Sat(\mathcal{F}, n)$ denotes the family of 
$n$-vertex $\mathcal{F}$-saturated graphs with exactly $\sat(\mathcal{F}, n)$ edges.

Hence, the classical Turán-type extremal problems and saturation problems may be viewed
as studying the two opposite extremal boundaries of the same class of graphs.

Faudree's survey~\cite{FFS2011} gives an overview of this area.
The study of \emph{saturation number} originated with complete graphs,
Erd\H{o}s, Hajnal and Moon~\cite{ErdosHajnalMoon1964} introduced the idea of \emph{saturation number} and determined $\sat(K_p,n)$ and $\Sat(K_p,n)$ for all integers $n$ and $p$. 
A natural extension is to consider the saturation problem for $(m+1)K_p$.
For the case $p=2$, Mader~\cite{Mader1973} first investigated and determined the structure of $(m+1)K_2$-saturated graphs. 
Based on Mader's work, K\'{a}szonyi and Tuza~\cite{KT1986} determined $\sat((m+1)K_2,n)$ and $\Sat((m+1)K_2,n)$ for $n\ge 3m$, 
later, Zhang, Lu and Yu~\cite{ZhangLuYu2023} strengthened this result.
For general $p$, Faudree~\cite{FAUDREE20095870} determined $\sat((m+1)K_p,n)$ and $\Sat((m+1)K_p,n)$ for sufficiently large $n$. 
More recently, Chen et al.~\cite{CHEN2024113868} studied the \emph{saturation number} of $K_p\cup (m-1)K_q$ for $m\ge 3$ and $2\le p<q$, 
while Zhu et al.~\cite{ZHU2025114530} further determined $\sat(K_p\cup (m-1)K_q,n)$ for $q\ge p\ge 2$ and $m\ge 2$, 
again both for sufficiently large $n$.

While determining the exact value of $\sat(\mathcal{F},n)$ is often difficult,
many studies instead aim to obtain estimates~\cite{HE2023148, CAO2023108}.
In particular, K\'{a}szonyi and Tuza~\cite{KT1986} proved that for every graph $H$ and sufficiently large $n$, 
there exists a constant $c=c(H)$ such that
\[
\sat(H,n)<cn.
\]

In this paper, we study \((m+1)K_3\)-saturated graphs. 
We assume that \(n\ge 3m+3\), since otherwise no graph obtained by adding a
single edge can contain \(m+1\) vertex-disjoint triangles.  Under this
assumption, an \((m+1)K_3\)-saturated graph cannot be complete, because
\(K_n\) itself contains \(m+1\) vertex-disjoint triangles.
Our main results are as follows.

We first introduce in Section~\ref{sec:upperBound} a family of
\((m+1)K_3\)-saturated graphs \(H(z)\), which gives a uniform
upper bound for all admissible values of \(m\) and \(n\).
\begin{theorem}\label{thm:saturation-bound}
For all integers \(m\ge 1\) and \(n\ge 3m+3\), 
$\sat((m+1)K_3,n)= O(n^{3/2}).$
\end{theorem}
We then improve the threshold in~\cite{FAUDREE20095870} for $(m+1)K_3$-saturated graphs.
\begin{theorem}\label{thm:saturation-number}
    For all integers \(m\ge 1\), if \(n\ge 9m+5\), then $\sat((m+1)K_3,n)=n+6m-1.$
\end{theorem}

Finally, we give a structural observation for saturated graphs with universal
vertices.  
For a graph \(H\), let \(\nu(H)\) denote the maximum number of pairwise
vertex-disjoint triangles in \(H\).
The structural observation can be stated as follows.

\begin{theorem}\label{thm:main-structural}
Let \(G\) be an \((m+1)K_3\)-saturated graph, and let \(Z\neq\emptyset\)
be the set of vertices of degree \(|V(G)|-1\).  Suppose that \(G-Z\) has at
least two components.  Then every component \(C\) of \(G-Z\) satisfies
$|V(C)|-3\nu(C)\ge 1.$
Moreover, if
$|V(C)|-3\nu(C)=1,$
then \(C\) is complete.
\end{theorem}

\section{Preliminaries}\label{sec:preliminary}

Throughout this paper, all graphs are finite, simple, and undirected.

For a positive integer \(t\), we write $[t]:=\{1,2,\ldots,t\}.$

Let \(G\) be a graph. We write \(V(G)\) and \(E(G)\) for the vertex set and
edge set of \(G\), respectively, and write $e(G):=|E(G)|.$
The complement of \(G\) is denoted by \(\overline G\). 
If a graph $H$ is isomorphic to $G$, we write \(G\cong H\).
If \(uv\notin E(G)\),
then \(G+uv\) denotes the graph obtained from \(G\) by adding the edge \(uv\). 
For \(X\subseteq V(G)\), we write \(G[X]\) for the
subgraph of \(G\) induced by \(X\), and write $G-X:=G[V(G)\setminus X].$
For two disjoint vertex sets \(X,Y\subseteq V(G)\), let \(e_G(X,Y)\) denote
the number of edges of \(G\) with one endpoint in \(X\) and the other endpoint
in \(Y\). We also write $e_G(X):=e(G[X]).$
The neighbourhood and closed neighbourhood of \(v\in V(G)\) are denoted by
$N_G(v):=\{u\in V(G):uv\in E(G)\},
N_G[v]:=N_G(v)\cup\{v\}.$
The degree of \(v\) in \(G\) is
$d_G(v):=|N_G(v)|$,
and the minimum degree of \(G\) is denoted by $\delta(G):=\min_{v\in V(G)}d_G(v).$
When the graph \(G\) is clear from the context, we simply write
\(N(v)\), \(N[v]\), and \(d(v)\).
Let $\mathcal C(G)$ denote the set of connected components of \(G\). When \(C\in\mathcal C(G)\),
we regard \(C\) as an induced subgraph of \(G\).
For pairwise vertex-disjoint graphs \(G_1,\ldots,G_s\), their union
\(G_1\cup\cdots\cup G_s\) is the graph with vertex set
\(\bigcup_{i=1}^s V(G_i)\) and edge set \(\bigcup_{i=1}^s E(G_i)\).
When we emphasize that this union is disjoint, we write
\(G_1\dot\cup\cdots\dot\cup G_s\).
For a positive integer \(t\) and a graph \(H\), we write \(tH\) for the
disjoint union of \(t\) copies of \(H\).

We denote by \(K_t\) the complete graph on \(t\) vertices and by
\(\overline K_t\) the empty graph on \(t\) vertices.

\begin{definition}[Join]
Let \(G\) and \(H\) be two vertex-disjoint graphs. The \emph{join} of \(G\)
and \(H\), denoted by
$G\otimes H,$
is the graph obtained from the disjoint union of \(G\) and \(H\) by adding all
edges between \(V(G)\) and \(V(H)\). Equivalently,
$V(G\otimes H)=V(G)\cup V(H)$
and
$E(G\otimes H)
=
E(G)\cup E(H)\cup \{uv:u\in V(G),\ v\in V(H)\}.$
\end{definition}

\begin{definition}[Packing]
Let \(G\) be a graph, and let \(\mathcal H\) be a family of graphs. An
\emph{\(\mathcal H\)-packing} in \(G\) is a collection
$\mathcal P=\{H_1,\ldots,H_t\}$
of pairwise vertex-disjoint subgraphs of \(G\) such that, for each
\(i\in[t]\), the graph \(H_i\) is isomorphic to some member of
\(\mathcal H\).

The maximum size of an \(\mathcal H\)-packing in \(G\) is called the
\emph{\(\mathcal H\)-packing number} of \(G\), and is denoted by
$\nu_{\mathcal H}(G).$
That is,
$\nu_{\mathcal H}(G)
=
\max\bigl\{
|\mathcal P|:
\mathcal P \text{ is an } \mathcal H\text{-packing in }G
\bigr\}.$
When \(\mathcal H=\{K_3\}\), we simply write
\(\nu(G):=\nu_{\{K_3\}}(G)\).
\end{definition}

\begin{definition}[Deficiency]
Let \(G\) be a graph. The \emph{\(K_3\)-packing deficiency} of \(G\) is
defined by
$\defi(G):=|V(G)|-3\nu(G).$
Equivalently, \(\defi(G)\) is the number of vertices left uncovered by a
maximum \(K_3\)-packing of \(G\). We say that \(G\) has a perfect
\(K_3\)-packing if \(\defi(G)=0\).
\end{definition}

\section{A Construction for $(m+1)K_3$-Saturated Graphs}\label{sec:upperBound}

In the early study of $(m+1)K_2$-saturated graphs, Mader~\cite{Mader1973} proved the following structural result.

\begin{theorem}[Mader~\cite{Mader1973}]\label{thm:Mader}
Let $G$ be a $(m+1)K_{2}$-saturated graph on $n$ vertices.
If $G$ is disconnected, then $G$ is a disjoint union of cliques, each of which has an odd number of vertices.
If $G$ is connected, then $G$ contains a vertex of degree $n-1$,
and the deletion of this vertex yields a $mK_{2}$-saturated graph.
\end{theorem}

For general $p\ge 3$, Faudree et al.~\cite{FAUDREE20095870} observed the following construction, which can be written as
\[
G(n,p,m+1):=
K_{p-2}\otimes\Bigl(mK_{p+1}\cup \overline{K}_{\,n-p(m+1)-m+2}\Bigr),
\]
where $\overline{K}_{\,n-p(m+1)-m+2}$ is the empty graph on $n-p(m+1)-m+2$ vertices.
They showed that this construction is extremal for sufficiently large $n$.

\begin{theorem}[Faudree et al.~\cite{FAUDREE20095870}]\label{thm:Faudree}
Let $m \ge 0$, $p\ge 3$, and
$n\ge p(p+1)(m+1)-p^2+2p-6$
be integers. Then
\[
\sat((m+1)K_p,n)
=
\left|E\!\left(G(n,p,m+1)\right)\right|
=
m\binom{p+1}{2}
+\binom{p-2}{2}
+(p-2)(n-p+2).
\]
\end{theorem}

We observe that, 
both in the connected case of Mader's constructions and in Faudree's construction, 
there exists at least one vertex adjacent to every other vertex of the graph. 
This motivates us to introduce the notion of a universal vertex. 

\begin{definition}[Universal vertex]
Let $G$ be a graph. A vertex $u\in V(G)$ is called a \emph{universal vertex} if
$N(u)=V(G)\setminus\{u\},$
or equivalently, if
$d(u)=|V(G)|-1.$
\end{definition}

Next, we introduce a family of \((m+1)K_3\)-saturated graphs \(H(z)\),
in which the parameter \(2z+1\) is precisely the number of universal vertices.
For every admissible value of \(z\), it gives $\sat((m+1)K_3,n)\le e(H(z)).$
This family includes, as the case \(z=0\), 
the construction of Faudree et al.~\cite{FAUDREE20095870} for \(p=3\).
In particular, when \(n\ge 12m+3\), 
the theorem of Faudree et al.~\cite{FAUDREE20095870} implies
$\sat((m+1)K_3,n)=e(H(0)).$

\subsection{A Family of Saturated Graphs}

We now describe the structure of $H(z)$.

For an integer \(z\) with \(0\le z\le m\), set
\(k:=n-3m\) and \(r:=k+z-1\).  Write
\(m-z=ar+b\), where \(0\le b<r\).
Define
\[
H(z):=
K_{2z+1}\otimes
\Bigl((r-b)K_{3a+1}\cup bK_{3a+4}\Bigr).
\]

To justify this construction, we first prove that a more general graph of the form
$K_{2z+1}\otimes \bigcup_{i=1}^r K_{3t_i+1}$
with $r \ge z+2$, 
is $(z + \sum_{i=1}^r t_i + 1)K_3$-saturated,
whose structure is illustrated in Figure~\ref{fig:structure1}.
The graph $H(z)$ defined above is a special case obtained by choosing the values of $t_1,\dots,t_r$ as evenly as possible.

  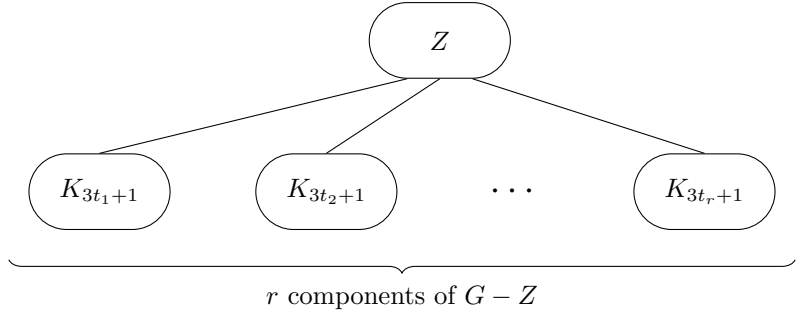
\begin{figure}[t]
\centering
\begin{tikzpicture}[
  box/.style={
    draw,
    rounded rectangle,
    minimum width=2.1cm,
    minimum height=1.0cm,
    align=center
  }
]

\node[box] (Z) at (6.0,2.0) {$Z$};

\node[box] (C1) at (1.5,0) {$K_{3t_1+1}$};
\node[box] (C2) at (4.5,0) {$K_{3t_2+1}$};
\node at (7.0,0) {\Large $\cdots$};
\node[box] (C3) at (9.5,0) {$K_{3t_{r}+1}$};

\draw (Z.south west) -- (C1.north);
\draw (Z.south) -- (C2.north);
\draw (Z.south east) -- (C3.north);

\draw[decorate, decoration={brace, mirror, amplitude=5pt}]
  (0.3,-0.9) -- (10.7,-0.9)
  node[midway, below=6pt] {$r$ components of $G-Z$};

\end{tikzpicture}
\caption{The graph
\(G=K_{2z+1}\otimes\bigcup_{i=1}^{r}K_{3t_i+1}\), where
\(Z=V(K_{2z+1})\).}
\label{fig:structure1}
\end{figure}

\begin{proposition}\label{prop:typical-construction}
Let $z\ge 0$ and let $\mathbf{t}=(t_1,\dots,t_r)$ be a sequence of nonnegative integers with
$r\ge z+2.$
Set
$m:=z+\sum_{i=1}^r t_i.$
Then $G = K_{2z+1}\otimes \bigcup_{i=1}^r K_{3t_i+1}$ is an $(m+1)K_3$-saturated graph.
\end{proposition}

\begin{proof}
Write \(Z:=V(K_{2z+1})\), and let \(C_i\cong K_{3t_i+1}\) denote the
\(i\)-th clique, for \(i\in[r]\).

We first show that
$\nu(G)=m$.

For the lower bound, each clique $C_i\cong K_{3t_i+1}$ contains $t_i$ pairwise disjoint triangles. Since $r\ge z+2$, in particular $r\ge z$, we may choose one vertex from each of $z$ distinct components and combine them with $2z$ vertices of $Z$ to form $z$ further triangles. Hence
\[
\nu(G)\ge z+\sum_{i=1}^r t_i=m.
\]

For the upper bound, let $\mathcal{P}$ be an arbitrary $K_3$-packing in $G$. 
Let $\tau$ be the number of triangles of $\mathcal{P}$ contained entirely in $Z$. For each $i\in [r]$, let
\begin{itemize}
\item $\alpha_i$ be the number of triangles of $\mathcal{P}$ contained entirely in $C_i$,
\item $\beta_i$ be the number of triangles of $\mathcal{P}$ containing exactly two vertices of $C_i$,
\item $\gamma_i$ be the number of triangles of $\mathcal{P}$ containing exactly one vertex of $C_i$.
\end{itemize}
Then
$3\alpha_i+2\beta_i+\gamma_i\le 3t_i+1$
and
$3\tau+\sum_{i=1}^r(\beta_i+2\gamma_i)\le 2z+1.$
Set
$e_i:=\alpha_i+\beta_i+\gamma_i-t_i$,
$e_i$ measures how many more triangles of $\mathcal{P}$ meet $C_i$ than the baseline number $t_i$ of disjoint triangles that can be packed entirely inside $C_i$.
Then
$3e_i=3\alpha_i+3\beta_i+3\gamma_i-3t_i\le 1+\beta_i+2\gamma_i,$
and hence
$e_i\le \left\lfloor\frac{1+\beta_i+2\gamma_i}{3}\right\rfloor.$
Therefore
\[
|\mathcal{P}|
=
\tau+\sum_{i=1}^r(\alpha_i+\beta_i+\gamma_i)
=
\tau+\sum_{i=1}^r(t_i+e_i)
\le
\tau+\sum_{i=1}^r t_i+\sum_{i=1}^r \left\lfloor\frac{1+\beta_i+2\gamma_i}{3}\right\rfloor.
\]
Since $\left\lfloor\frac{1+s}{3}\right\rfloor\le \frac{s}{2}$
for all integers $s \ge 0$,
we obtain
\[
|\mathcal{P}|
\le
\tau+\sum_{i=1}^r t_i+\frac12\sum_{i=1}^r(\beta_i+2\gamma_i)
\le
\tau+\sum_{i=1}^r t_i+\frac12(2z+1-3\tau).
\]
Thus
$|\mathcal{P}|
\le
\sum_{i=1}^r t_i+z+\frac{1-\tau}{2}.$
Since $|\mathcal{P}|$ is an integer, it follows that
$|\mathcal{P}|\le \sum_{i=1}^r t_i+z=m.$

Combining the upper and lower bounds, we obtain
$\nu(G)=m.$

Now we prove that \(G\) is \((m+1)K_3\)-saturated.
Let \(u,v\) be a non-edge of \(G\). Since any vertex in \(Z\) is adjacent to all vertices
and each \(C_i\) is complete, we may assume that
\(u\in C_a\) and \(v\in C_b\) for two distinct indices \(a,b\in[r]\).

In \(G+uv\), choose a vertex \(z_0\in Z\). Then \(\{u,v,z_0\}\) is a
triangle. We now construct \(m\) further disjoint triangles avoiding
\(\{u,v,z_0\}\).

For each \(i\in[r]\), the clique \(C_i\cong K_{3t_i+1}\) contains \(t_i\)
pairwise disjoint triangles leaving any prescribed vertex uncovered. We
leave \(u\) uncovered in \(C_a\), leave \(v\) uncovered in \(C_b\), and
leave one arbitrary vertex uncovered in every other component. This gives
\(\sum_{i=1}^r t_i\) disjoint triangles.

After using \(z_0\), there remain \(2z\) vertices in \(Z\). Pair them into
\(z\) pairs. Since \(r\ge z+2\), there are at least \(z\) components other
than \(C_a\) and \(C_b\). Taking the uncovered vertices from \(z\) such
components together with the \(z\) pairs of vertices in \(Z\), we obtain
\(z\) additional disjoint triangles.

Thus \(G+uv\) contains
$1+z+\sum_{i=1}^r t_i=m+1$
pairwise vertex-disjoint triangles. By definition, \(G\) is \((m+1)K_3\)-saturated.
This completes the proof.
\end{proof}

It remains to show that, for every integer $n\ge 3m+3$, set $k = n - 3m \ge 3$, 
the above construction can indeed be carried out with suitable parameters. 
That is, we need to choose $z$ and nonnegative integers $t_1,\dots,t_{k+z-1}$ so that
\[
G = K_{2z+1} \otimes \bigcup_{i=1}^{k+z-1} K_{3t_i+1}
\]
has order $n$. Once this is done, Proposition~\ref{prop:typical-construction} implies that $G$ is an $(m+1)K_3$-saturated graph.

\begin{lemma}\label{lem:existence_Z_structure} 
    Let \(n,m\) be integers with \(n\ge 3m+3\), and set \(k:=n-3m\). 
    Then for every integer \(z\) with \(0\le z\le m\), there exist nonnegative integers \(t_1,\ldots,t_{k+z-1}\) such that \[ G=K_{2z+1}\otimes \bigcup_{i=1}^{k+z-1}K_{3t_i+1} \] has order \(n\) and is \((m+1)K_3\)-saturated. 
    Moreover, $|\mathcal C(G-Z)|=k+z-1\ge z+2$, where \(Z=V(K_{2z+1})\). 
\end{lemma}

\begin{proof}
Since \(n\ge 3m+3\), we have \(k=n-3m\ge 3\).
Fix an integer \(z\) with \(0\le z\le m\), and set \(r:=k+z-1\).
Then \(r\ge z+2\).  Since \(m-z\ge 0\), we may choose nonnegative integers
\(t_1,\dots,t_r\) such that
\[
    \sum_{i=1}^r t_i=m-z.
\]

For $  G=K_{2z+1}\otimes\bigcup_{i=1}^r K_{3t_i+1}$,
$|V(G)| = (2z+1)+\sum_{i=1}^r(3t_i+1)=2z+1+3(m-z)+(k+z-1)=3m+k=n$.
Moreover,
$|\mathcal C(G-Z)|=r=k+z-1\ge z+2.$

Hence, the chosen parameters satisfy
$z+\sum_{i=1}^r t_i=z+(m-z)=m$
and \(r\ge z+2\). Proposition~\ref{prop:typical-construction}
implies that \(G\) is an \((m+1)K_3\)-saturated graph.
\end{proof}

\subsection{Analysis of the Number of Edges}
Next we compute the number of edges of \(H(z)\).

Fix an integer \(z\) with \(0\le z\le m\), 
and consider graphs of the form 
$G=K_{2z+1}\otimes\bigcup_{i=1}^{k+z-1}K_{3t_i+1}$, with $\sum_{i=1}^{k+z-1}t_i=m-z.$ 
We first compute the number of edges of such a graph, 
and then minimize it with respect to \(t_1,\ldots,t_{k+z-1}\).

Since every vertex in $K_{2z+1}$ has degree $n-1$, and every vertex in a component $K_{3t_i+1}$ has degree
$3t_i+(2z+1),$
we have
\[
\begin{aligned}
2|E(G)|
&=
\sum_{v\in V(G)} d(v) \\
&=
(2z+1)(n-1)+\sum_{i=1}^{k+z-1}(3t_i+1)(3t_i+2z+1) \\
&= (2n-2)z + n-1 + (6z+6)\sum_{i=1}^{k+z-1} t_i + (2z+1)(k+z-1) + 9\sum_{i=1}^{k+z-1} t_i^2 \\
&= (2n-2)z + (n-1) + (6z+6)\left(\frac{n-k}{3}-z\right) + (2z+1)(k+z-1) + 9\sum_{i=1}^{k+z-1} t_i^2 \\
&=-4z^2+(4n-9)z+3n-k-2+9\sum_{i=1}^{k+z-1} t_i^2.
\end{aligned}
\]

Let
$S:=\sum_{i=1}^{k+z-1} t_i=\frac{n-k}{3}-z.$
For fixed $z$, the quantity $\sum t_i^2$ is minimized when the $t_i$ are as equal as possible. Write
$a:=\left\lfloor \frac{S}{k+z-1}\right\rfloor$, $b:=S-a(k+z-1)$
so that $0\le b<k+z-1.$
Then
\[
\min_{t_i}\sum_{i=1}^{k+z-1} t_i^2
=
b(a+1)^2+(k+z-1-b)a^2
=
\frac{S^2-b^2}{k+z-1}+b.
\]
Hence we have 
\[
\begin{aligned}
  &\min_{t_i}2|E(G)|  \\ 
  &= -4z^2 + (4n-9)z + 3n - k -2  + \frac{n^2 + k^2 - 2nk + 9z^2 - 6nz + 6kz}{k+z-1} + 9(b - \frac{b^2}{k+z-1})\\
  &= -4z^2 + (4n-9)z + 3n - k - 2 + 9z - 6n - 3k + 9 + \frac{n^2 + 4k^2 + 4nk - 6n - 12k + 9}{z+k-1} + 9(b - \frac{b^2}{k+z-1}) \\
  &= -4z^2 + 4nz + \frac{n^2 + 4k^2 + 4nk - 6n - 12k + 9}{z+k-1} - 3n - 4k + 7 + 9(b - \frac{b^2}{k+z-1}) \\
  &= -4z^2 + 4nz + \frac{(n-3)^2 + 4k(n-3) + 4k^2}{z+k-1} - 3n - 4k + 7 + 9(b - \frac{b^2}{k+z-1}) \\
  & = -4z^2 + 4nz + \frac{(n-3+2k)^2}{z+k-1} - 3n - 4k + 7 + 9(b - \frac{b^2}{k+z-1}). \\
\end{aligned}
\]
Write 
\[g(z) =  -4z^2 + 4nz + \frac{(n-3+2k)^2}{z+k-1} - 3n - 4k + 7 , \quad R(z) = 9(b - \frac{b^2}{k+z-1}).\]
Then $e(H(z)) = \frac 12 (g(z) + R(z))$,
Furthermore, we have the following trivial estimate for $R(z)$.
\begin{equation}\label{eq:es-for-R}
    0 \le R(z) \le R(z) |_{b = \frac{k+z-1}{2}} = \frac{9(k+z-1)}{4}.
\end{equation}

Now we can derive an upper bound by choosing $z=0$. Since $k = n - 3m$, 
\begin{equation}\label{eq:e-H0-formula}
\begin{aligned}
  2e(H(0)) &= \frac{(n + 2k -3)^2}{k-1} - 3n - 4k + 7 + R(0)\\
  &= \frac{(n-3+2(n-3m))^2}{n-3m-1} - 3n - 4(n-3m) + 7 + R(0) \\
    &= 9\frac{(n-3m-1 + m)^2}{n-3m-1} - 7n + 12m + 7 + R(0) \\
  &= 2n + 3m - 2 + 9\frac{m^2}{n-3m-1} + R(0).\\
\end{aligned}
\end{equation}

By~\eqref{eq:es-for-R}, we obtain
\begin{equation}\label{eq:saturation-number-bound-2}
    \sat((m+1)K_3, n) \le e(H(0)) \le \frac{17}{8}n + \frac{9m^2}{2(n-3m-1)} - \frac{15}{8}m - \frac{17}{8}.
\end{equation}

We now prove Theorem~\ref{thm:saturation-bound}.
\begin{proof}[proof to Theorem~\ref{thm:saturation-bound}]
We distinguish two cases.

First suppose that $k\ge \sqrt n$. 
By~\eqref{eq:e-H0-formula}, $2e(H(0))=2n+3m-2+\frac{9m^2}{n-3m-1}+R(0).$
Since $m\le n/3$ and $k=n-3m\ge \sqrt n$, we have
$\frac{m^2}{n-3m-1}
=
\frac{m^2}{k-1}
=
O(n^{3/2}).$
Moreover $R(0)\le 9(k-1)/4=O(n)$. Hence
\[
e(H(0))=O(n^{3/2}).
\]

Now suppose that \(k<\sqrt n\).  For \(n\ge 16\), we have
$m=\frac{n-k}{3}>\frac{n-\sqrt n}{3}\ge \sqrt n.$
Hence \(z:=\lfloor\sqrt n\rfloor\) satisfies \(0\le z\le m\).  The finitely
many admissible pairs with \(n<16\) may be absorbed into the implicit
constant.
Using the formula
$2e(H(z))
=
-4z^2+4nz+\frac{(n-3+2k)^2}{z+k-1}-3n-4k+7+R(z),$
together with
$0\le R(z)\le \frac{9(k+z-1)}4,$
we obtain
\[
2e(H(z))
\le
4nz+\frac{(n-3+2k)^2}{z+k-1}+O(n).
\]
Since $z=O(\sqrt n)$, $k<\sqrt n$, and $z+k-1=\Omega(\sqrt n)$, the right-hand
side is $O(n^{3/2})$. Thus
\[
e(H(z))=O(n^{3/2}).
\]

In both cases, the construction gives an $(m+1)K_3$-saturated graph with
$O(n^{3/2})$ edges. Therefore
$\sat((m+1)K_3,n)= O(n^{3/2}),$
as desired.
\end{proof}
In contrast to~\eqref{eq:saturation-number-bound-2}, the upper bound given in
Theorem~\ref{thm:saturation-bound} is independent of \(m\) and depends only on
\(n\).

\section{Saturation Number for Disjoint Triangles for \(n \ge 9m + 5\)}

We first prove two useful lemmas.

\begin{lemma}
\label{lem:degree-one-implies-universal}
Let $G$ be an $(m+1)K_3$-saturated graph. If
$\delta(G)=1,$
then $G$ contains a universal vertex.
\end{lemma}

\begin{proof}
Let $u$ be a vertex of degree $1$, and write
$N_G(u)=\{z\}.$
Take any vertex
$x\in V(G)\setminus\{u,z\}.$
Since $z$ is the unique neighbour of $u$, we have $xu\notin E(G).$
By the $(m+1)K_3$-saturation of $G$, the graph $G+xu$ contains an
$(m+1)K_3$-packing. Since $G$ itself contains no such packing, this packing
must use the new edge $xu$.
Thus the new edge $xu$ lies in a triangle, say $xuw$,
with $uw \in E(G)$.
But $u$ has only one neighbour $z$ in $G$. Hence
$w=z.$
Therefore $xz\in E(G).$

Since $x\in V(G)\setminus\{u,z\}$ was arbitrary, the vertex $z$ is adjacent
to every vertex other than itself. Hence $z$ is a universal vertex.
\end{proof}

\begin{lemma}\label{lem:standard-upper-bound}
Let \(m\ge 1\) and \(n\ge \max \{4m+1, 3m+3\}\). If \(G\in \Sat((m+1)K_3,n)\), then $e(G)\le n+6m-1.$
\end{lemma}

\begin{proof}
Consider
\[
    H(0)=K_1\otimes\left(mK_4\cup \overline K_{n-4m-1}\right).
\]
Since \(n\ge 4m+1\), this graph is well-defined. By
Proposition~\ref{prop:typical-construction}, \(H(0)\) is
\((m+1)K_3\)-saturated. Moreover,
$e(H(0))=(n-1)+m\binom{4}{2}=n+6m-1.$
Therefore
$\sat((m+1)K_3,n)\le e(H(0))=n+6m-1.$
\end{proof}

\subsection{Minimum-Degree Restrictions for Extremal Graphs}

Now we determine the minimum degree of an $(m+1)K_3$-saturated graph for a certain range.

\begin{proposition}
\label{prop:large-n-delta-at-most-two}
Let \(m\ge 1\) and \(n\ge 6m+5\). If
$G\in \Sat((m+1)K_3,n),$
then
$\delta(G)\le 2.$
Moreover, \(G\) has no isolated vertex, and hence
$1\le \delta(G)\le 2.$
\end{proposition}

\begin{proof}

First, \(G\) has no isolated vertex. Indeed, if \(v\) were isolated, then
for any \(u\neq v\), the added edge \(uv\) would not lie in any triangle
of \(G+uv\), contradicting \((m+1)K_3\)-saturation. Hence
$\delta(G)\ge 1.$

Suppose for contradiction that
$\delta(G)\ge 3.$
Let \(v\) be a vertex of minimum degree, and write
$d:=d_G(v)=\delta(G)\ge 3.$
For every
$u\in V(G)\setminus N_G[v],$
the edge \(uv\) is missing. By saturation, \(G+uv\) contains an
\((m+1)K_3\)-packing, and this packing must use the new edge \(uv\).
Therefore \(uv\) lies in a triangle of \(G+uv\). Hence \(u\) has at least
one neighbour in \(N_G(v)\).
It follows that there are at least
$n-d-1$
edges between \(N_G(v)\) and \(V(G)\setminus N_G[v]\). Since \(v\) sends
\(d\) edges to \(N_G(v)\), we have
\[
\sum_{w\in N_G(v)} d_G(w)\ge (n-d-1)+d=n-1.
\]
Moreover, every vertex in \(V(G)\setminus N_G[v]\) has degree at least \(d\).
Thus
\begin{equation}\label{eq:degree-count-delta-ge-3}
\begin{aligned}
2e(G)
&=
d_G(v)+\sum_{w\in N_G(v)}d_G(w)
+\sum_{u\in V(G)\setminus N_G[v]}d_G(u) \\
&\ge
d+(n-1)+d(n-d-1) \\
&=
d(n-d)+n-1.
\end{aligned}
\end{equation}

We next show that, in all possible cases for \(d\), the assumption
\(d=\delta(G)\ge 3\) forces
$e(G)\ge 2n-5.$

(i) If \(d\le n-3\), since \(d\ge 3\),
\[
d(n-d)-3(n-3)=(d-3)(n-d-3)\ge 0.
\]
Therefore \(d(n-d)\ge 3(n-3)\). By~\eqref{eq:degree-count-delta-ge-3}, 
we obtain \(2e(G)\ge 3(n-3)+n-1=4n-10\), and hence \(e(G)\ge 2n-5\).

(ii) If \(d=n-2\), then
$2e(G)=\sum_{u\in V(G)}d_G(u)\ge n\delta(G)=n(n-2).$
Moreover,
$n(n-2)-(4n-10)=(n-3)^2+1>0.$
Hence
\[
2e(G)\ge n(n-2)>4n-10,
\]
and so again
\[
e(G)>2n-5.
\]

(iii) The case \(d=n-1\) is impossible, otherwise $G$ is complete.

Thus in all cases with \(\delta(G)\ge 3\), we obtain
$e(G)\ge 2n-5.$
Since \(n\ge 6m+5\),
$e(G) \ge 2n-5\ge n+6m>n+6m-1,$
contradicting lemma~\ref{lem:standard-upper-bound}.

Therefore
$\delta(G)\le 2.$
Together with \(\delta(G)\ge 1\), this gives
$1\le \delta(G)\le 2.$
\end{proof}

\begin{proposition}
\label{prop:large-n-excludes-delta-two}
Let \(m\ge 1\) and let
$n\ge 9m+5.$
If
$G\in \Sat((m+1)K_3,n),$
then
$\delta(G)\neq 2.$
\end{proposition}

\begin{proof}

By proposition~\ref{prop:large-n-delta-at-most-two}, $1 \le \delta(G) \le 2$.
 Suppose for contradiction that
$\delta(G)=2.$
Let \(v\in V(G)\) be a vertex of degree \(2\), and write
$N_G(v)=\{x,y\}.$
Set
$W:=V(G)\setminus\{v,x,y\}.$

For every \(w\in W\), the edge \(vw\) is missing. By the
\((m+1)K_3\)-saturation of \(G\), the graph \(G+vw\) contains an
\((m+1)K_3\)-packing. Since \(G\) itself contains no such packing, this
packing must use the new edge \(vw\). Hence \(vw\) lies in a triangle of
\(G+vw\). As the only neighbours of \(v\) in \(G\) are \(x\) and \(y\), it
follows that every vertex of \(W\) is adjacent to at least one of \(x\) and
\(y\).

Let
$r:=|\{w\in W: wx\in E(G)\text{ and }wy\in E(G)\}|,$
and
\[
\varepsilon=
\begin{cases}
1,& \text{if }xy\in E(G),\\
0,& \text{if }xy\notin E(G).
\end{cases}
\]
Then
$e_G(W,\{x,y\})=n-3+r.$
Therefore
\[
e(G)
=
2+\varepsilon+e_G(W,\{x,y\})+e_G(W)
=
n-1+r+\varepsilon+e_G(W).
\]

Now define
$D:=\{w\in W: w\text{ is adjacent to exactly one of }x,y\}.$
Then
$|D|=n-3-r.$
Every vertex \(w\in D\) is adjacent to exactly one of \(x\) and \(y\), and
is not adjacent to \(v\). Since \(\delta(G)=2\), each such vertex has at
least one neighbour in \(W\).

We claim that \(G[W]\) contains at least \(m-\varepsilon\) pairwise
vertex-disjoint triangles.

Indeed, if \(\varepsilon=1\), then \(xy\in E(G)\). Choose any \(w\in W\).
The edge \(vw\) is missing, so \(G+vw\) contains an \((m+1)K_3\)-packing
using the edge \(vw\). The triangle containing \(vw\) has the form
$vws$
for some \(s\in\{x,y\}\). After deleting this triangle, the remaining
\(m\) triangles are vertex-disjoint and lie in
$G-\{v,w,s\}.$
Among these \(m\) triangles, at most one can use the remaining vertex of
\(\{x,y\}\setminus\{s\}\). Hence at least \(m-1=m-\varepsilon\) of them
are contained in \(G[W]\).

If \(\varepsilon=0\), then \(xy\notin E(G)\). By saturation, \(G+xy\)
contains an \((m+1)K_3\)-packing using the new edge \(xy\). The triangle
containing \(xy\) has the form
$xyt$
for some \(t\in \{v\}\cup W\). Removing this triangle leaves \(m\)
pairwise vertex-disjoint triangles in \(G-\{x,y,t\}\). If \(t=v\), then
these \(m\) triangles are all contained in \(G[W]\). If \(t\in W\), then
the vertex \(v\) is isolated in \(G-\{x,y,t\}\), since
\(N_G(v)=\{x,y\}\). Thus none of the remaining \(m\) triangles uses \(v\),
and again all of them are contained in \(G[W]\). Therefore \(G[W]\)
contains at least \(m=m-\varepsilon\) pairwise vertex-disjoint triangles.

Fix such an \((m-\varepsilon)K_3\)-packing in \(G[W]\), and let
\(S\) be the set of vertices covered by this packing. 
Every vertex in \(D\setminus S\) has at least one neighbour in \(W\). The
edges witnessing this are not among the \(3(m - \varepsilon)\) triangle edges, and each
such additional edge can be incident with at most two vertices of
\(D\setminus S\). Hence

\[
e_G(W)
\ge
3(m - \varepsilon)+
\left\lceil
\frac{\max\{|D|-3(m - \varepsilon),0\}}{2}
\right\rceil = 3(m-\varepsilon)
+
\left\lceil
\frac{\max\{n-3-r-3(m-\varepsilon),0\}}{2}
\right\rceil.
\]

We obtain
\begin{equation}\label{eq:lower-bound-prop2}
\begin{aligned}
e(G)
&\ge
n-1+r+\varepsilon
+
3(m-\varepsilon)
+
\left\lceil
\frac{\max\{n-3-r-3(m-\varepsilon),0\}}{2}
\right\rceil \\
&=
n+3m-1+r-2\varepsilon
+
\left\lceil
\frac{\max\{n-r-3m-3+3\varepsilon,0\}}{2}
\right\rceil.
\end{aligned}
\end{equation}

We shall use the following elementary inequality: for all integers
\(M\ge 0\) and \(r\ge 0\),
\begin{equation}\label{eq:inequality-pro2}
    r+\left\lceil\frac{\max\{M-r,0\}}{2}\right\rceil
\ge
\left\lceil\frac{M}{2}\right\rceil.
\end{equation}
Apply this with
$M =n-3m-3+3\varepsilon$.
Since \(n\ge 9m+5\), we have \(M\ge 0\).
Thus~\eqref{eq:lower-bound-prop2} and~\eqref{eq:inequality-pro2} imply
\begin{equation}\label{eq:lower-boound-prop2-2}
e(G) \ge n+3m-1-2\varepsilon +
\left\lceil \frac{n-3m-3+3\varepsilon}{2} \right\rceil.
\end{equation}

Finally, since \(n\ge 9m+5\) and \(\varepsilon\in\{0,1\}\), 
which implies that 
$\left\lceil \frac{2+3\varepsilon}{2} \right\rceil = 1+2\varepsilon$, 
we have
\[
\left\lceil
\frac{n-3m-3+3\varepsilon}{2}
\right\rceil
\ge
\left\lceil
\frac{6m+2+3\varepsilon}{2}
\right\rceil
=
3m+1+2\varepsilon.
\]
Substituting this into \eqref{eq:lower-boound-prop2-2}, we get
$e(G)\ge n+3m-1-2\varepsilon+3m+1+2\varepsilon = n+6m.$
This contradicts Lemma~\ref{lem:standard-upper-bound}. Therefore
\(\delta(G)\neq 2\). Together with
Proposition~\ref{prop:large-n-delta-at-most-two}, we obtain
\(\delta(G)=1\).
\end{proof}

\subsection{A sharper threshold for \texorpdfstring{$(m+1)K_3$}{(m+1)K3}-saturation}

In this section, we mainly follow the argument of Faudree et al.~\cite{FAUDREE20095870} and prove the following theorem in our framework.

\begin{lemma}\label{lemma:triangle-case-sharper-threshold}
Let \(m\ge 1\), \(n\ge 3m + 3\) and let $G$ be a $(m+1)K_3$-saturated graph. 
If there is a vertex of degree $1$,
then $e(G) \ge n+6m-1.$
\end{lemma}

\begin{proof}
Let \(v\in V(G)\) be a vertex of degree \(1\). By
Lemma~\ref{lem:degree-one-implies-universal}, the unique neighbour of \(v\)
is a universal vertex. Denote this universal vertex by \(z\). Thus
$N_G(v)=\{z\}.$
Set
$W:=V(G)\setminus\{z,v\}.$

For every \(w\in W\), the edge \(vw\) is missing. Since \(G\) is
\((m+1)K_3\)-saturated, \(G+vw\) contains an \((m+1)K_3\)-packing. In such
a packing, the vertex \(v\) must lie in the triangle
$\{v,z,w\},$
since in \(G+vw\) the only neighbours of \(v\) are \(z\) and \(w\).
Therefore, after deleting this triangle, the graph
$G[W\setminus\{w\}]$
contains an \(mK_3\)-packing. In particular, \(G[W]\) contains \(m\)
pairwise vertex-disjoint triangles.

Fix such an \(mK_3\)-packing in \(G[W]\), and denote its triangles by
$T_1,T_2,\ldots,T_m.$
Let
\[
    T:=V(T_1)\dot\cup V(T_2)\dot\cup\cdots\dot\cup V(T_m),
    \qquad
    R:=W\setminus T.
\]

We first record a simple consequence of saturation.

\begin{claim}\label{claim:lower-extra-neighbour}
For every \(i\in[m]\) and every vertex \(a\in V(T_i)\), the vertex \(a\)
has a neighbour outside \(V(T_i)\cup\{z\}\).
\end{claim}

\begin{proof}
Suppose not. Then for some \(i\in[m]\) and some \(a\in V(T_i)\), all
neighbours of \(a\) lie in \(V(T_i)\cup\{z\}\).

Choose a vertex \(x\in V(T_i)\setminus\{a\}\). Since \(xv\notin E(G)\),
the graph \(G+xv\) contains an \((m+1)K_3\)-packing. The vertex \(v\) must
lie in the triangle
$\{v,z,x\}.$
After deleting this triangle, the remaining graph contains an
\(mK_3\)-packing \(P\) avoiding \(v,z,x\).

We claim that \(a\notin V(P)\). Indeed, after deleting \(z\) and \(x\), the
vertex \(a\) has at most one remaining neighbour, namely the third vertex
of \(T_i\). Hence \(a\) cannot lie in any triangle of \(P\).

Therefore \(P\) is an \(mK_3\)-packing in \(G\) avoiding \(z,x,a\). But
\(\{z,x,a\}\) is a triangle in \(G\), because \(z\) is universal and
\(x,a\in V(T_i)\). Thus
$P\cup\{\{z,x,a\}\}$
is an \((m+1)K_3\)-packing in \(G\), a contradiction.
\end{proof}

Now define \(F\) to be the graph with vertex set \(T\) and edge set
$E(F):=E(G[T])\setminus \bigcup_{i=1}^m E(T_i).$
Thus \(F\) consists exactly of the edges of \(G[T]\) which are not among the
fixed triangle edges of \(T_1,\ldots,T_m\).

We shall prove that
\[
    e(F)+e_G(T,R)\ge 3m.
\]
To this end, observe that every connected component of $F$ either contains a cycle
or is a tree. In the latter case, it either has an edge to $R$ or has no
edge to $R$. We first exclude connected tree components of \(F\)
which have no edge to \(R\).

\begin{claim}\label{claim:lower-no-bad-tree-component}
No connected component \(C\) of \(F\) is a tree satisfying
$e_G(V(C),R)=0.$
\end{claim}

\begin{proof}
Suppose, to the contrary, that such a connected component \(C\) exists. By
Claim~\ref{claim:lower-extra-neighbour}, \(C\) is not an isolated vertex.
Hence \(C\) is a tree with at least two vertices.

Choose a longest path in \(C\), and let \(u_1\) be one of its endvertices.
Let \(w\) be the unique neighbour of \(u_1\) in \(C\). Suppose
$u_1\in V(T_1)$, $T_1=\{u_1,u_2,u_3\}.$
Since \(u_1w\in E(F)\), the vertex \(w\) lies in a triangle different from
\(T_1\). Write
$ T_2=\{w,s,t\}.$

We first show that
$u_2w,u_3w\in E(G).$
Suppose \(u_3w\notin E(G)\). Since \(vu_2\notin E(G)\), the graph
\(G+vu_2\) contains an \((m+1)K_3\)-packing. The vertex \(v\) must lie in
the triangle
$\{v,z,u_2\}.$
After deleting this triangle, the remaining graph contains an
\(mK_3\)-packing \(P\) in \(G[W\setminus\{u_2\}]\).

Since \(C\) has no edge to \(R\), and since \(u_1\) is a leaf of \(C\), the
only possible neighbours of \(u_1\) in \(G[W\setminus\{u_2\}]\) are
\(u_3\) and \(w\). But \(u_3w\notin E(G)\), so \(u_1\) lies in no triangle
of \(G[W\setminus\{u_2\}]\). Hence \(u_1\notin V(P)\). Consequently,
$P\cup\{\{z,u_1,u_2\}\}$
is an \((m+1)K_3\)-packing in \(G\), a contradiction. Therefore
\(u_3w\in E(G)\).

By the same argument with \(u_2\) and \(u_3\) interchanged, using the
missing edge \(vu_3\) instead of \(vu_2\), we also obtain \(u_2w\in E(G)\).
Thus
$u_2w,u_3w\in E(G).$

Now return to the longest path chosen above. Since this path starts with
the edge \(u_1w\), at most one of \(u_2,u_3\) can be the vertex following
\(w\) on this path. Relabel \(u_2,u_3\), so that \(u_2\) is
not that vertex. We claim that \(u_2\) is also a leaf of \(C\).

Indeed, if \(u_2\) had a neighbour in \(C\) other than \(w\), then in the
connected component of \(C-w\) containing \(u_2\), choose a vertex \(y\) farthest from
\(u_2\). The path from \(y\) to the other end of the original longest path
would be strictly longer than the original longest path, because the branch
through \(u_2\) has length at least \(2\) from \(w\), while the branch
through \(u_1\) has length \(1\) from \(w\). This is impossible. Hence
\(u_2\) is a leaf of \(C\), and its unique neighbour in \(F\) is \(w\).
See figure~\ref{fig:u2-leaf}.

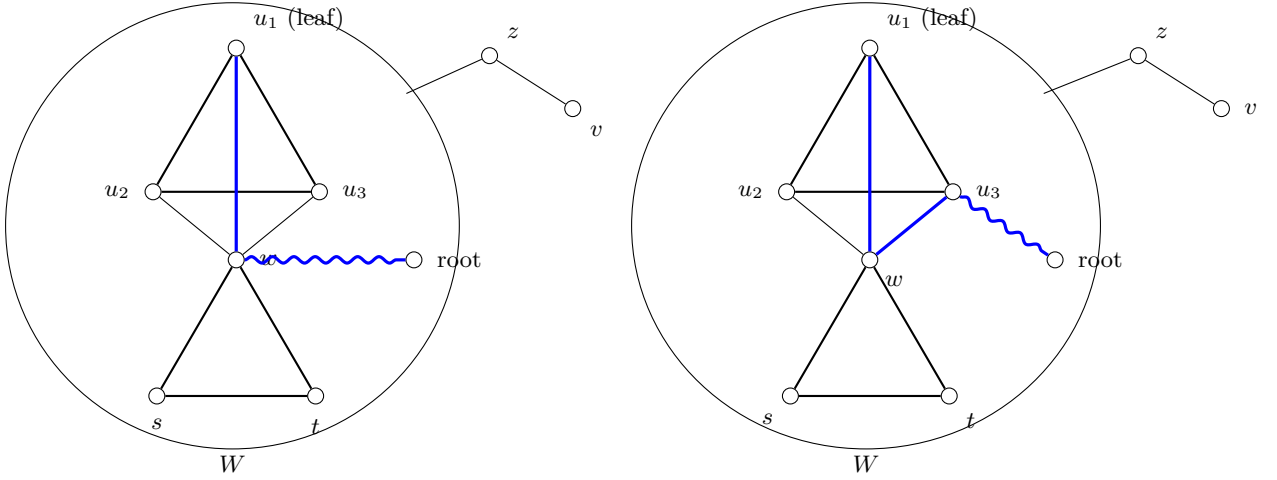
\begin{figure}[t]
\centering

\begin{subfigure}[t]{0.48\textwidth}
\centering
\begin{tikzpicture}[
    vertex/.style={circle, draw, inner sep=1.5pt, minimum size=6pt},
    every node/.style={font=\small}
]

\node[vertex] (u1) at (0,2.25) {};
\node[vertex] (u2) at (-1.10,0.35) {};
\node[vertex] (u3) at (1.10,0.35) {};
\node[vertex] (w)  at (0,-0.55) {};
\node[vertex] (s)  at (-1.05,-2.35) {};
\node[vertex] (t)  at (1.05,-2.35) {};
\node[vertex] (root) at (2.35,-0.55) {};

\node[vertex] (z) at (3.35,2.15) {};
\node[vertex] (v) at (4.45,1.45) {};

\draw (-0.05,-0.10) ellipse [x radius=3.00cm, y radius=2.95cm];
\node at (-0.05,-3.25) {$W$};

\node[above right=1pt of u1] {$u_1$ (leaf)};
\node[left=2pt of u2] {$u_2$};
\node[right=2pt of u3] {$u_3$};
\node[right=2pt of w] {$w$};
\node[below=2pt of s] {$s$};
\node[below=2pt of t] {$t$};
\node[right=2pt of root] {root};
\node[above right=1pt of z] {$z$};
\node[below right=1pt of v] {$v$};

\draw[thick] (u1) -- (u2);
\draw[thick] (u1) -- (u3);
\draw[thick] (u2) -- (u3);

\draw[blue, very thick] (u1) -- (w);
\draw (u2) -- (w);
\draw (u3) -- (w);

\draw[thick] (w) -- (s);
\draw[thick] (w) -- (t);
\draw[thick] (s) -- (t);

\draw[blue, very thick, decorate,
      decoration={snake, amplitude=1.2pt, segment length=8pt}]
    (w) -- (root);

\coordinate (xW) at (2.25,1.65);
\draw (xW) -- (z);
\draw (z) -- (v);

\end{tikzpicture}
\end{subfigure}
\hfill
\begin{subfigure}[t]{0.48\textwidth}
\centering
\begin{tikzpicture}[
    vertex/.style={circle, draw, inner sep=1.5pt, minimum size=6pt},
    every node/.style={font=\small}
]

\node[vertex] (u1) at (0,2.25) {};
\node[vertex] (u2) at (-1.10,0.35) {};
\node[vertex] (u3) at (1.10,0.35) {};
\node[vertex] (w)  at (0,-0.55) {};
\node[vertex] (s)  at (-1.05,-2.35) {};
\node[vertex] (t)  at (1.05,-2.35) {};
\node[vertex] (root) at (2.45,-0.55) {};

\node[vertex] (z) at (3.55,2.15) {};
\node[vertex] (v) at (4.65,1.45) {};

\draw (-0.05,-0.10) ellipse [x radius=3.10cm, y radius=2.95cm];
\node at (-0.05,-3.25) {$W$};

\node[above right=1pt of u1] {$u_1$ (leaf)};
\node[left=2pt of u2] {$u_2$};
\node[right=2pt of u3] {$u_3$};
\node[below right=0pt of w] {$w$};
\node[below left=1pt of s] {$s$};
\node[below right=1pt of t] {$t$};
\node[right=2pt of root] {root};
\node[above right=1pt of z] {$z$};
\node[right=2pt of v] {$v$};

\draw[thick] (u1) -- (u2);
\draw[thick] (u1) -- (u3);
\draw[thick] (u2) -- (u3);

\draw[blue, very thick] (u1) -- (w);
\draw (u2) -- (w);
\draw[blue, very thick] (w) -- (u3);

\draw[thick] (w) -- (s);
\draw[thick] (w) -- (t);
\draw[thick] (s) -- (t);

\draw[blue, very thick, decorate,
      decoration={snake, amplitude=1.2pt, segment length=8pt}]
    (u3) -- (root);

\coordinate (xW) at (2.30,1.65);
\draw (xW) -- (z);
\draw (z) -- (v);

\end{tikzpicture}
\end{subfigure}

\caption{ \(u_2\) is also a leaf of \(C\). The blue edges indicate the chosen longest path.}
\label{fig:u2-leaf}
\end{figure}

Since \(u_1\) is a leaf of \(C\), its only neighbour in
\(T\setminus V(T_1)\) is \(w\). Also \(C\) has no edge to \(R\). Therefore
$u_1s\notin E(G).$
By saturation, \(G+u_1s\) contains an \((m+1)K_3\)-packing. Let \(Q\) be
the triangle of this packing containing the new edge \(u_1s\), and let
\(\mathcal P\) be the remaining \(m\) triangles. Then \(\mathcal P\) is an
\(mK_3\)-packing in \(G\) avoiding \(V(Q)\).

The third vertex of \(Q\) must be a common neighbour of \(u_1\) and \(s\)
in \(G\). Since
$N_G(u_1)\subseteq \{z,u_2,u_3,w\},$
and since \(u_2s\notin E(G)\) because \(u_2\) is a leaf of \(C\), the third
vertex of \(Q\) belongs to \(\{z,w,u_3\}\). Hence
\[
    Q\in
    \bigl\{
        \{u_1,s,z\},
        \{u_1,s,w\},
        \{u_1,s,u_3\}
    \bigr\}.
\]

We now consider these three possibilities.

First suppose
$Q=\{u_1,s,z\}.$
The only possible neighbours of \(u_2\) outside \(V(Q)\) are \(u_3\) and
\(w\). Hence, if \(u_2\in V(\mathcal P)\), then \(\mathcal P\) contains the
triangle
$\{u_2,u_3,w\}.$
In this case,
$\bigl(\mathcal P\setminus\{\{u_2,u_3,w\}\}\bigr)
    \cup
    \{\{u_1,u_2,u_3\},\{z,w,s\}\}$
is an \((m+1)K_3\)-packing in \(G\), a contradiction. Therefore
\(u_2\notin V(\mathcal P)\). But then
$\mathcal P\cup\{\{z,u_1,u_2\}\}$
is an \((m+1)K_3\)-packing in \(G\), again a contradiction.

Next suppose
$Q=\{u_1,s,w\}.$
The only possible neighbours of \(u_2\) outside \(V(Q)\) are \(z\) and
\(u_3\). Hence, if \(u_2\in V(\mathcal P)\), then \(\mathcal P\) contains
the triangle
$\{z,u_2,u_3\}.$
In this case,
$\bigl(\mathcal P\setminus\{\{z,u_2,u_3\}\}\bigr)
    \cup
    \{\{u_1,u_2,u_3\},\{z,w,s\}\}$
is an \((m+1)K_3\)-packing in \(G\), a contradiction. Therefore
\(u_2\notin V(\mathcal P)\). But then
$\mathcal P\cup\{\{u_1,u_2,w\}\}$
is an \((m+1)K_3\)-packing in \(G\), again a contradiction.

Finally suppose
$Q=\{u_1,s,u_3\}.$
The only possible neighbours of \(u_2\) outside \(V(Q)\) are \(z\) and
\(w\). Hence, if \(u_2\in V(\mathcal P)\), then \(\mathcal P\) contains the
triangle
$\{z,u_2,w\}.$
In this case,
$\bigl(\mathcal P\setminus\{\{z,u_2,w\}\}\bigr)
    \cup
    \{\{u_1,u_2,u_3\},\{z,w,s\}\}$
is an \((m+1)K_3\)-packing in \(G\), a contradiction. Therefore
\(u_2\notin V(\mathcal P)\). But then
$\mathcal P\cup\{\{u_1,u_2,u_3\}\}$
is an \((m+1)K_3\)-packing in \(G\), again a contradiction.

All cases lead to contradictions. Therefore no such component \(C\) exists.
\end{proof}

We now count edges. Let \(\mathcal A\) be the set of components of \(F\)
which contain a cycle, and let \(\mathcal B\) be the set of tree components
of \(F\). For every \(A\in\mathcal A\), since \(A\) is connected and
contains a cycle,
$e_F(A)\ge |V(A)|.$
For every \(B\in\mathcal B\), we have
$e_F(B)=|V(B)|-1.$
Moreover, by Claim~\ref{claim:lower-no-bad-tree-component}, each such
\(B\) has at least one edge to \(R\). Choosing one such edge for each tree
component, and noting that these chosen edges are distinct, we get
\[
\begin{aligned}
    e(F)+e_G(T,R)
    &\ge
    \sum_{A\in\mathcal A} e_F(A)
    +
    \sum_{B\in\mathcal B}\bigl(e_F(B)+1\bigr)  \\
    &\ge
    \sum_{A\in\mathcal A}|V(A)|
    +
    \sum_{B\in\mathcal B}|V(B)|  \\
    &= |T|=3m.
\end{aligned}
\]
The fixed triangles \(T_1,\ldots,T_m\) contribute exactly \(3m\) edges.
Therefore
\[
\begin{aligned}
    e_G(W)
    &\ge 3m+e(F)+e_G(T,R) \\
    &\ge 3m+3m \\
    &=6m.
\end{aligned}
\]

Finally, since \(z\) is universal and \(v\) has no neighbour in \(W\), we
have
$e(G)=d_G(z)+e_G(W)=(n-1)+e_G(W).$
Thus
$e(G)\ge n-1+6m=n+6m-1.$
\end{proof}

Now we prove Theorem~\ref{thm:saturation-number}.
\begin{proof}
    Let $G \in \Sat((m+1)K_3, n)$. 
    Since $n \ge 9m +5$,  by Proposition~\ref{prop:large-n-excludes-delta-two},
    $\delta(G) = 1$.
    Let $v$ be a vertex of degree $1$,
    by Lemma~\ref{lemma:triangle-case-sharper-threshold}, $e(G) \ge n+6m-1$.
    On the other hand,
    by Lemma~\ref{lem:standard-upper-bound}, $e(G) \le n + 6m -1$.
    Then $\sat((m+1)K_3, n) = e(G) = n+6m-1.$
\end{proof}

\section{Deficiency of Components in \(G-Z\)}

In this section, we prove Theorem~\ref{thm:main-structural}.  Let \(Z\)
be the set of universal vertices of an \((m+1)K_3\)-saturated graph \(G\).
We study the \(K_3\)-packing deficiency
$|V(C)|-3\nu(C)$
of each component \(C\) of \(G-Z\).  We first show that every such component
has positive deficiency, and then prove that, when \(G-Z\) has at least two
components, a component with deficiency one must be complete.

\begin{lemma}[Lower bound on component deficiency]\label{lem:component-deficiency}
Let $G$ be an $(m+1)K_3$-saturated graph, and let $Z\neq \emptyset$ be the set of universal vertices of $G$.
Then every connected component $C$ of $G-Z$ satisfies
\[
|V(C)|-3\nu(C)\ge 1.
\]
Equivalently, no connected component of $G-Z$ admits a perfect $K_3$-packing.
\end{lemma}

\begin{proof}
Suppose to the contrary that some connected component $C$ of $G-Z$ admits a perfect $K_3$-packing.
Write
\[
|V(C)|=3t,
\qquad
\nu(C)=t
\]
for some integer $t\ge 1$.

Since $C$ is a component of $G-Z$, the graph $G$ is not complete, then $G$ has a non-edge $xy$.
Because $G$ is $(m+1)K_3$-saturated, the graph $G+xy$ contains an $(m+1)K_3$-packing.
At most one triangle in such a packing uses the added edge $xy$, so by deleting that triangle we obtain an $mK_3$-packing in $G$.
Since $G$ itself is $(m+1)K_3$-free, it follows that
$\nu(G)=m.$
Moreover, since $G$ has a non-edge and $G+xy$ contains an $(m+1)K_3$-packing, necessarily
$|V(G)|\ge 3m+3.$

We now distinguish two cases.

\medskip
\noindent\textbf{Case 1.} $G-Z$ has at least two connected components.

Let $D\neq C$ be another connected component of $G-Z$, and choose $x\in V(C)$, $y\in V(D)$.
Then
$xy\notin E(G).$
Fix an $(m+1)K_3$-packing $P'$ in $G+xy$.

Let $Q\subseteq P'$ be the set of all triangles in $P'$ that intersect $V(C)$, and write
$q:=|Q|.$
Define
\[
S:=V(Q)\setminus V(C),
\qquad
s:=|S|.
\]

We claim that
$S\subseteq Z\cup\{y\}.$
Indeed, in $G+xy$, the only edge joining $C$ and $D$ is the added edge $xy$.
Hence at most one triangle of $Q$ can meet $D$, namely the triangle containing $xy$.
Any other triangle of $Q$ that is not contained entirely in $C$ must use vertices outside $C$ coming from the universal set $Z$.
Thus the claim follows.

Since
$S\subseteq Z\cup\{y\},$
and $Z$ is a clique of universal vertices, it follows that $G[S]$ is a clique.
Therefore
$\nu(G[S])=\left\lfloor \frac{s}{3}\right\rfloor.$

On the other hand, each triangle in $Q$ has three vertices, and the number of vertices of $Q$ lying in $C$ is at most $|V(C)|=3t$.
Hence
$3q = |V(Q)\cap V(C)| + |S| \le 3t+s,$
so
$q\le t+\frac{s}{3}.$
Since $q$ is an integer, we obtain
$q\le t+\left\lfloor \frac{s}{3}\right\rfloor.$

Now, $C$ has a perfect $K_3$-packing, so $G[V(C)]$ contains $t$ pairwise vertex-disjoint triangles.
Together with the
$\left\lfloor \frac{s}{3}\right\rfloor$
pairwise vertex-disjoint triangles in $G[S]$, we obtain
$t+\left\lfloor \frac{s}{3}\right\rfloor$
pairwise vertex-disjoint triangles in $G[V(C)\cup S]$.
Hence
\[
\nu(G[V(C)\cup S])\ge t+\left\lfloor \frac{s}{3}\right\rfloor \ge q.
\]

Since $P'\setminus Q$ is vertex-disjoint from $V(C)\cup S$, we deduce that
\[
\begin{aligned}
\nu(G)
&\ge \nu(G[V(C)\cup S])+\nu(P'\setminus Q) \\
&\ge q + (m+1-q) \\
&= m+1,
\end{aligned}
\]
contradicting $\nu(G)=m$.

\medskip
\noindent\textbf{Case 2.} $G-Z$ is connected.

Then
$C=G-Z.$
Since $C$ admits a perfect $K_3$-packing and $Z$ induces a clique, $G$ contains at least
$t+\left\lfloor \frac{|Z|}{3}\right\rfloor$
pairwise vertex-disjoint triangles. Therefore
$\nu(G)\ge t+\left\lfloor \frac{|Z|}{3}\right\rfloor.$
Using $\nu(G)=m$ and $|V(G)|=3t+|Z|$, we obtain
\[
\begin{aligned}
|V(G)|-3m
&= |V(G)|-3\nu(G) \\
&\le (3t+|Z|)-3\left(t+\left\lfloor \frac{|Z|}{3}\right\rfloor\right) \\
&= |Z|-3\left\lfloor \frac{|Z|}{3}\right\rfloor \\
&\in \{0,1,2\}.
\end{aligned}
\]
On the other hand, we already know that $|V(G)|\ge 3m+3$, i.e.,
$|V(G)|-3m\ge 3,$
a contradiction.

Both cases lead to contradictions. Therefore every connected component $C$ of $G-Z$ satisfies
$|V(C)|-3\nu(C)\ge 1.$
\end{proof}

\begin{lemma}[Completeness of deficiency-$1$ components]\label{lemma:completenessOf1compnenent}
Let $G$ be an $(m+1)K_3$-saturated graph, and let $Z$ be the set of
universal vertices of $G$. If $C\in \mc(G-Z)$ satisfies $\defi(C)=1$,
then $C$ is complete.
\end{lemma}

\begin{proof}
Suppose for contradiction that $C$ is not complete. 
Then there exist two vertices $x,y\in V(C)$ such that
$xy\notin E(G).$
Since $G$ is $(m+1)K_3$-saturated, adding the edge $xy$ creates an $(m+1)K_3$-packing in $G+xy$.
Fix such a packing and denote it by $P'$.

Let $Q\subseteq P'$ be the set of all triangles that intersect $C$, and write
$|Q|=q.$
Then
$\nu(P'\setminus Q)=m+1-q.$
Set
$S:=V(Q)\setminus V(C).$

Since $C$ is a connected component of $G-Z$, every vertex in $S$ must belong to $Z$. Hence
$S\subseteq Z.$
Moreover, each triangle in $Q$ contributes exactly three vertices, so
$3q=|V(Q)|.$
Since
$V(Q)=\bigl(V(Q)\cap V(C)\bigr)\,\dot\cup\,\bigl(V(Q)\setminus V(C)\bigr),$
we obtain
$3q=|V(Q)|\le |V(C)|+|S|,$
and therefore
\begin{equation}\label{ieq:lemma6-1}
    q\le \left\lfloor \frac{|V(C)|+|S|}{3}\right\rfloor.
\end{equation}

We next show that $\nu(G[C\cup S]) \ge \left\lfloor \frac{|V(C)|+|S|}{3}\right\rfloor$.
Write
$\nu(C)=t.$, 
by $\defi(C) = 1$, 
$|V(C)|=3t+1.$
Let $\mathcal T$ be a maximum $K_3$-packing of $C$. 
Then $|\mathcal T|=t$, and $\mathcal T$ covers exactly $3t$ vertices of $C$, so there is a unique vertex $u\in V(C)$ uncovered by $\mathcal T$.

Now write $|S|=3a+r$, $r\in\{0,1,2\}$.
Since $S\subseteq Z$, the graph $G[S]$ is complete, and every vertex of $S$ is adjacent to every vertex of $C$.

\medskip
\noindent
\textbf{Case 1:} $r=0$.

In this case, $S$ contains $a$ pairwise vertex-disjoint triangles. Together with $\mathcal T$, we obtain
$\nu(G[C\cup S])\ge t+a
=
\left\lfloor \frac{3t+1+3a}{3}\right\rfloor
=
\left\lfloor \frac{|V(C)|+|S|}{3}\right\rfloor.$

\medskip
\noindent
\textbf{Case 2:} $r=1$.

Again, $S$ contains $a$ pairwise vertex-disjoint triangles. Together with $\mathcal T$, we obtain
$\nu(G[C\cup S])\ge t+a
=
\left\lfloor \frac{3t+1+3a+1}{3}\right\rfloor
=
\left\lfloor \frac{|V(C)|+|S|}{3}\right\rfloor.$

\medskip
\noindent
\textbf{Case 3:} $r=2$.

Choose $a$ pairwise vertex-disjoint triangles in $S$, and let $s_1,s_2$ be the two remaining vertices of $S$.
Since $s_1,s_2\in S\subseteq Z$ and $u\in V(C)$, we have
\[
us_1,\,us_2,\,s_1s_2\in E(G).
\]
Thus, $us_1s_2$ forms a triangle. Together with $\mathcal T$ and the $a$ triangles chosen from $S$, we obtain
$\nu(G[C\cup S])\ge t+a+1
=
\left\lfloor \frac{3t+1+3a+2}{3}\right\rfloor
=
\left\lfloor \frac{|V(C)|+|S|}{3}\right\rfloor.$

\medskip
In all three cases,
\begin{equation}\label{ieq:lemma6-2}
    \nu(G[C\cup S])\ge \left\lfloor \frac{|V(C)|+|S|}{3}\right\rfloor.
\end{equation}

Note that $P'\setminus Q$ is vertex-disjoint from $C\cup S$. Hence
$\nu(G) \ge
\nu\Bigl(G\bigl[(V(C)\cup S)\,\dot\cup\,V(P'\setminus Q)\bigr]\Bigr)\ge
\nu(G[C\cup S])+(m+1-q)$,
combining~\eqref{ieq:lemma6-1} and~\eqref{ieq:lemma6-2},
we obtain
\[
\nu(G)
\ge
\left\lfloor \frac{|V(C)|+|S|}{3}\right\rfloor +(m+1-q)
\ge m+1,
\]
contradicting the assumption that $\nu(G)=m$.

Therefore, $C$ must be complete.
\end{proof}

\begin{remark}
As a consequence, if $|\mathcal C(G-Z)|=1$, then no component $C$ of $G-Z$
can satisfy $\defi(C)=1$. Indeed, if such a component existed, then by
Lemma~\ref{lemma:completenessOf1compnenent}, $C$ would be complete.
Then $Z = V(G)$, $\mc(G-Z) = \emptyset$.
\end{remark}

The preceding two lemmas together give the proof of Theorem~\ref{thm:main-structural}

\begin{proof}
Since $|\mc(G-Z)|\ge 2$, for any component $C \in \mc(G-Z)$,
$|V(C)| - 3\nu(C) \ge 1$ follows from Lemma~\ref{lem:component-deficiency}.
If $|V(C)| - 3\nu(C) = 1$, the completeness of $C$ follows from Lemma~\ref{lemma:completenessOf1compnenent}.
\end{proof}

\section{Conclusion}

In this paper, we studied saturation numbers for disjoint triangles. 
We constructed a family of
\((m+1)K_3\)-saturated graphs which gives the uniform upper bound
$\sat((m+1)K_3,n)=O(n^{3/2})$
for all \(m\ge 1\) and \(n\ge 3m+3\), with an absolute implicit constant.
We also proved the exact value
$\sat((m+1)K_3,n)=n+6m-1$
for \(n\ge 9m+5\), improving the previously known large-order threshold in
the triangle case in~\cite{FAUDREE20095870}. Finally, we established structural restrictions on the
components of \(G-Z\), where \(Z\) is the set of universal vertices, showing
that such components must have positive \(K_3\)-packing deficiency and that
deficiency-one components are complete. These restrictions are consistent
with the structure of the extremal saturated graphs arising from our
construction, suggesting that the same exact formula may remain valid beyond
the range established here.

\bibliographystyle{plain} 
\bibliography{refs}  
  

\end{document}